\documentclass[lettersize,journal]{IEEEtran}
\usepackage{amsmath,amsfonts,amssymb,amsthm}
\usepackage{algorithm}
\usepackage{algpseudocode}
\usepackage{array}
\usepackage{textcomp}
\usepackage{subcaption}
\usepackage[hidelinks]{hyperref}
\usepackage{verbatim}
\usepackage{graphicx}
\usepackage{cite}
\usepackage{bm}
\usepackage{booktabs}
\usepackage{xcolor}
\usepackage{siunitx}
\usepackage{tikz}
\usepackage{flushend}

\definecolor{accentcol}{rgb}{0.0000, 0.4470, 0.7410} 
\definecolor{accentcol2}{rgb}{0.8500, 0.3250, 0.0980} 

\usetikzlibrary{shapes}

\DeclareSIUnit\decibelmilliwatt{dBm}

\DeclareMathOperator*{\diag}{diag}
\DeclareMathOperator*{\rank}{rank}

\DeclareMathOperator*{\sgn}{sgn}
\DeclareMathOperator*{\minimize}{minimize}

\renewcommand{\vec}[1]{{\mathbf{#1}}}
\newcommand{\mat}[1]{{\mathbf{#1}}}

\newcommand{\RR}{\mathbb{R}}

\renewcommand{\Im}{\operatorname{Im}}
\newtheorem{theorem}{Theorem}[section]
\newtheorem{proposition}[theorem]{Proposition}

\begin{document}

\title{Single-Source Localization as an Eigenvalue Problem}

\author{Martin~Larsson,
        Viktor~Larsson,
        Kalle~Åström,
        and~Magnus~Oskarsson
\thanks{M. Larsson is with the Centre for Mathematical Sciences, Lund University, Lund, Sweden, and Combain Mobile AB, Lund, Sweden (e-mail: martin.larsson@math.lth.se).}
\thanks{V. Larsson, K. Åström, and M. Oskarsson are with the Centre for Mathematical Sciences, Lund University, Lund, Sweden (e-mail: viktor.larsson@math.lth.se, karl.astrom@math.lth.se, magnus.oskars\-son\-@math.lth.se).}}

\markboth{IEEE Transactions on signal processing,~Vol.~73, 2025}%
{Larsson \MakeLowercase{\textit{et al.}}: Single-Source Localization as an Eigenvalue Problem}

\maketitle

\begin{abstract}
    This paper introduces a novel method for solving the single-source localization problem, specifically addressing the case of trilateration.
    We formulate the problem as a weighted least-squares problem in the squared distances and demonstrate how suitable weights are chosen to accommodate different noise distributions.
    By transforming this formulation into an eigenvalue problem, we leverage existing eigensolvers to achieve a fast, numerically stable, and easily implemented solver.
    Furthermore, our theoretical analysis establishes that the globally optimal solution corresponds to the largest real eigenvalue, drawing parallels to the existing literature on the trust-region subproblem.
    Unlike previous works, we give special treatment to degenerate cases, where multiple and possibly infinitely many solutions exist. We provide a geometric interpretation of the solution sets and design the proposed method to handle these cases gracefully.
    Finally, we validate against a range of state-of-the-art methods using synthetic and real data, demonstrating how the proposed method is among the fastest and most numerically stable.
\end{abstract}


\section{Introduction}

\IEEEPARstart{T}{he} single-source localization problem has received a considerable amount of attention due to its broad application within areas such as GNSS positioning \cite{petropoulos_chapter_2021}, wireless networks \cite{li_review_2020}, molecular conformations \cite{dunitz_distance_1990}, robot kinematics \cite{geng_3-2-1_1994}, and indoor positioning \cite{ma_wi-fi_2022}. The problem consists of finding an unknown receiver position using estimates of distances to known sender locations, also called anchors or landmarks. These distances are often acquired using approaches based on Time Of Arrival (TOA), Received Signal Strength (RSS), or Time Difference Of Arrival (TDOA). In all cases, some signal, e.g., radio, sound, or light, is emitted from the senders and received at the sought position, or vice versa. 

In the case of TOA, the one-way propagation time of the signal is measured, and knowing the propagation speed in the medium, the distance traveled can be calculated. RSS instead utilizes the fact that signals get attenuated in the medium, and by modeling this attenuation, a distance measurement can be derived from the received signal strength.
Solving the single-source localization problem for the cases of TOA and RSS is called \emph{trilateration}.
TDOA is similar to TOA, but the former features an additional unknown offset in the distance measurement, typically due to unsynchronized clocks, that needs to be estimated. This fundamentally changes the localization problem and is referred to as \emph{multilateration}. In this paper, we focus solely on trilateration.

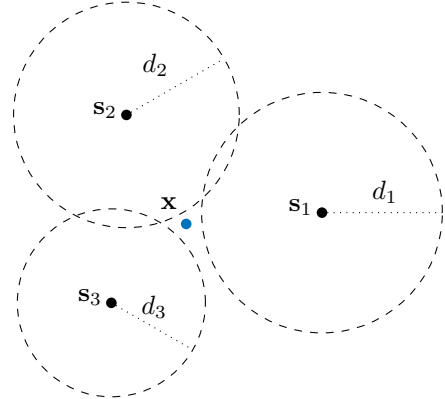
\begin{figure}
	\centering
	\begin{tikzpicture}[scale=1.0, every node/.style={scale=1.0},
		dot/.style = {circle, fill, minimum size=#1,
			inner sep=0pt, outer sep=0pt},
		dot/.default = 4pt,  
		circ/.style = {circle, draw, dashed, minimum size=#1,
			inner sep=0pt, outer sep=0pt}, 
		]
		\node[dot, label={[above left]$\vec{x}$}, color=accentcol] (x) at (-0.0051,0.0496) {};
		\node[dot, label={[left]$\vec{s}_1$}] (s1) at (1.8,0.2) {};
		\node[dot, label={[left]$\vec{s}_2$}] (s2) at (-0.8,1.5) {};
		\node[dot, label={[left]$\vec{s}_3$}] (s3) at (-1.0,-1.0) {};
		\node[circ=3.2cm] at (s1) {};
		\node[circ=3.0cm] at (s2) {};
		\node[circ=2.5cm] at (s3) {};
        \draw[dotted] (s1) -- +(0:0.5*3.2cm) node[pos=0.5, above] {$d_1$};
        \draw[dotted] (s2) -- +(30:0.5*3.0cm) node[pos=0.5, above left] {$d_2$};
        \draw[dotted] (s3) -- +(-30:0.5*2.5cm) node[pos=0.5, above] {$d_3$};
	\end{tikzpicture}
	\caption{Three senders $\vec{s}_1, \vec{s}_2, \vec{s}_3$ and the maximum likelihood solution $\vec{x}$ to the trilateration problem. The circles indicate the corresponding distance measurements $d_1$, $d_2$, and $d_3$.}
	\label{fig:trilat}
\end{figure}

The problem of trilateration can be formalized as having $m$ senders $\vec{s}_j \in \RR^n$, $j=1,\dotsc,m$, of known position with distance measurements $d_j \in \RR$ to an unknown receiver position $\vec{x} \in \RR^n$.
This can be modeled as
\begin{equation}
    \label{eq:toa_model}
	d_j = \|\vec{x}-\vec{s}_j\| + \epsilon_j,
\end{equation}
where $\epsilon_j$ represents additive noise in the measurements.
See \figurename~\ref{fig:trilat} for an example in 2D with three senders.
Provided i.i.d.\ Gaussian noise, the Maximum Likelihood (ML) estimator for $\vec{x}$ is given by solving
\begin{equation}
	\label{eq:ls}
	\minimize_\vec{x} \sum_{j=1}^m (\|\vec{x}-\vec{s}_j\| - d_j)^2.
\end{equation}
This is a nonlinear, nonsmooth, and nonconvex optimization problem that can have multiple local minima. As a result, proposed approaches have so far been limited to iterative methods \cite{luke_simple_2017, beck_iterative_2008, jyothi_solvit_2020, sirola_closed-form_2010, chan_exact_2006} or have relied on some form of relaxation of the problem.
Common for the iterative methods is that, while they under some circumstances guarantee convergence to a stationary point of \eqref{eq:ls}, they do not guarantee convergence to a global minimum and are thus dependent on a reasonable initialization.
A convex relaxation of the problem was proposed in \cite{beck_exact_2008}, where it was solved using Semidefinite Programming (SDP), and another convex formulation was given in \cite{ismailova_penalty_2016}, where the problem was solved using a variation of the Convex-Concave Procedure (CCP) \cite{lipp_variations_2016}.
However, due to these relaxations, there is no guarantee that the solution is optimal in the original cost function \eqref{eq:ls}.

A different approach is given by minimizing the error in the squared distances
\begin{equation}
	\label{eq:sls}
	\minimize_\vec{x} \sum_{j=1}^m (\|\vec{x}-\vec{s}_j\|^2 - d_j^2)^2.
\end{equation}
This formulation has the benefit of being polynomial and smooth, although, it lacks the statistical interpretation that \eqref{eq:ls} has and gives higher influence to larger errors and distances.
To mitigate this, a weighting of the terms can be introduced to better approximate \eqref{eq:ls} \cite{cheung_least_2004, chen_achieving_2013}.
This can be extended further by employing an Iteratively Reweighted Least Squares (IRLS) scheme that indeed converges to \eqref{eq:ls} \cite{beck_iterative_2008, ismailova_improved_2015}, or by modifying the weights based on uncertainties in the sender positions \cite{chen_reaching_2014}.
Common for these weighted approaches is that they all assume Gaussian additive noise in the distances.
A more detailed comparison of the two cost functions \eqref{eq:ls} and \eqref{eq:sls} can be found in \cite{larsson_accuracy_2010,chen_achieving_2013}.

A solution to \eqref{eq:sls} was provided in \cite{beck_exact_2008}, where they considered the equivalent problem
\begin{equation}
	\label{eq:sls_constrained}
	\minimize_{\vec{x},\alpha} \sum_{j=1}^m (\alpha -2\vec{x}^T \vec{s}_j + \vec{s}_j^T\vec{s}_j - d_j^2)^2,
\end{equation}
under the constraint $\alpha = \vec{x}^T\vec{x}$. This is a Generalized Trust Region Subproblem (GTRS) which was reduced to a single equation in one variable and subsequently solved using bisection.
A closed-form approximate solution to \eqref{eq:sls} was proposed in \cite{zhou_closed-form_2011}, where they introduced the artificial constraint $\sum_j \|\vec{x}-\vec{s}_j\|^2 = \sum_j d_j^2$, which is not satisfied in general, resulting in suboptimal solutions. In the minimal case, when $m=n$, an exact solution can be found and several closed-form methods have been proposed \cite{thomas_revisiting_2005, manolakis_efficient_1996, coope_reliable_2000}.
Various linear methods have also been proposed, some of which are equivalent to the unconstrained version of \eqref{eq:sls_constrained}, see e.g. \cite{caffery_new_2000, navidi_statistical_1998, stoica_lecture_2006, zekavat_handbook_2012.ch2}.

In this paper, we present a novel method for solving the trilateration problem, specifically, the weighted version of \eqref{eq:sls}, by transforming it into an eigenvalue problem.
Our main contributions here are threefold.
First, in Section~\ref{sec:weighting}, we show how to derive suitable weights to more closely approximate \eqref{eq:ls}, and in contrast to previous works, we demonstrate how our approach generalizes to a wider range of noise distributions, e.g., log-normal noise.
Second, in Section~\ref{sec:eigenvalue_formulation}, we show that the problem of finding the stationary points of the weighted cost function \eqref{eq:approx_cost} can be formulated as an eigenvalue problem, and in particular, we prove that the global minimum corresponds to the largest real eigenvalue of a certain matrix.
Third, we treat degenerate cases, where the trilateration problem has multiple solutions, an area that has not received sufficient attention in the literature.
The proposed method is summarized in Algorithm~\ref{alg:proposed}, and in Sections~\ref{sec:experiment_synth} and \ref{sec:experiment_real}, we validate it against several state-of-the-art methods on both real and synthetic data.
The material presented here is partially based on the conference paper \cite{larsson_optimal_2019} and the manuscript \cite[Paper~VIII]{larsson_localization_2022}\footnote{The manuscript in \cite{larsson_localization_2022} is an earlier unpublished version of this paper.}.

\section{The Weighted Cost Function}
\label{sec:weighting}
In the presence of noise, minimizing the cost function in \eqref{eq:sls} is not the same as minimizing the one in \eqref{eq:ls}. However, with a suitable weighting of the terms, the former provides an accurate approximation of the latter. We generalize this and show how weights can be derived for a wider range of noise distributions.

We start by introducing the residual functions
\begin{equation}
	\label{eq:res}
	r_j(\vec{x}) = \Psi_j(\|\vec{x}-\vec{s}_j\|^2) - \Psi_j(d_j^2),
\end{equation}
for $j=1,\dotsc,m$, where $d_j$ are noisy distance measurements and $\Psi_j(z)$ are normalization transformations differentiable at $z = d_j^2$.
Let $\vec{r}(\vec{x}) \in \RR^m$ denote the residual vector with elements $r_j(\vec{x})$.
The idea is to choose $\Psi_j(z)$ such that the noise in the normalized measurements $\Psi_j(d_j^2)$ is Gaussian, and, consequently, minimizing the residuals $r_j(\vec{x})$ in the least squares sense yields an ML-estimate for $\vec{x}$.
Formally, we seek to minimize
\begin{equation}
	\label{eq:opt_cost}
	h_0(\vec{x}) = \vec{r}(\vec{x})^T \mat{P} \vec{r}(\vec{x}),
\end{equation}
where $\mat{P}^{-1} \in \RR^{m \times m}$ denotes the known covariance matrix of $\vec{r}(\vec{x})$.
Note that if $\mat{P}=\mat{I}$ and we choose $\Psi_j(z) = z$, we get the optimization problems in \eqref{eq:sls}. Similarly, choosing $\Psi_j(z) = \sqrt{z}$ yields the problem in \eqref{eq:ls} provided $d_j \geq 0$.

Depending on $\Psi_j(z)$, minimizing $h_0(\vec{x})$ can be a difficult problem. We will therefore, in each residual $r_j(\vec{x})$, replace $\Psi_j(z)$ with its first order Taylor approximation at $d_j^2$.
\begin{equation}
	\Psi_j(z) \approx \Psi_j(d_j^2) + \Psi_j'(d_j^2)(z-d_j^2).
\end{equation}
Insertion into \eqref{eq:res} yields the approximate residuals
\begin{equation}
	\label{eq:approx_res}
    \tilde{r}_j(\vec{x}) = \Psi_j'(d_j^2) \left( \|\vec{x}-\vec{s}_j\|^2-d_j^2 \right).
\end{equation}
Naturally, this approximation becomes more accurate as $d_j^2$ approaches $\|\vec{x}-\vec{s}_j\|^2$, and in noiseless case, the approximation is exact.
Now, let $\mat{W}$ be a positive definite weighting matrix with elements $w_{ij} = \Psi_i'(d_i^2) P_{ij} \Psi_j'(d_j^2)$. Note that $\mat{W}$ is constant and does not depend on $\vec{x}$.
The cost function approximating $h_0(\vec{x})$ can now be written as
\begin{equation}
	\label{eq:approx_cost}
	h(\vec{x}) = \frac{1}{4}\sum_{i=1}^m \sum_{j=1}^m w_{ij} (\|\vec{x}-\vec{s}_i\|^2 - d_i^2)(\|\vec{x}-\vec{s}_j\|^2 - d_j^2),
\end{equation}
where we have added a factor of $1/4$ for later convenience.
This is the function for which the proposed method finds the global minimum.

\subsection{Different Noise Distributions}
\label{sec:noise_distributions}
Different forms of measurement noise are accommodated by finding suitable normalization transformations $\Psi_j(z)$.
For example, when working with TOA measurements, the distances $d_j$ are often modeled to contain Gaussian additive noise. In this case, letting $\Psi_j(z) = \sqrt{z}$ will cause $\Psi_j(d_j^2) = d_j$ to be Gaussian assuming $d_j > 0$. The weighting in the approximate residuals \eqref{eq:approx_res} is then given by
\begin{equation}
	\Psi_j'(z) = \frac{1}{2\sqrt{z}} \quad \Rightarrow \quad \Psi_j'(d_j^2) =  \frac{1}{2d_j}.
\end{equation}
In particular, if $\mat{P} = \mat{I}$, then $\mat{W}$ is diagonal with elements $w_{jj} = 1/4d_j^2$.
The same weighting was derived in \cite{cheung_least_2004}, and in \cite{chen_achieving_2013} it was shown to result in solutions reaching the Cramer-Rao Lower Bound (CRLB) accuracy.

A different example is when radio RSS measurements are used. Then the signal strength can be modeled using the log-distance path loss model \cite[Chapter~8]{sharp_wireless_2019}, also known as the one-slope model \cite[Chapter~4.7]{cost_action_231_1999}, given by
\begin{equation}
    \label{eq:log-distance_path_loss}
	C_j = (C_0)_j - 10 \eta_j \log_{10}(\|\vec{x}-\vec{s}_j\|) + \epsilon_j,
\end{equation}
where $(C_0)_j$ and $\eta_j$ are known parameters and $\epsilon_j$ is zero mean Gaussian noise. Given an RSS measurement $C_j$ the corresponding distance measurement $d_j$ can be calculated as
\begin{equation}
    \label{eq:rss_distance_measurement}
    d_j^2 = 10^{\frac{(C_0)_j - C_j}{5\eta_j}}
    = 10^{-\frac{\epsilon_j}{5\eta_j}} \|\vec{x}-\vec{s}_j\|^2.
\end{equation}
Letting $\Psi_j(z) = 5 \eta_j \log_{10}(z)$, we get that $\Psi_j(d_j^2)$ is Gaussian and the weighting in the approximate residuals \eqref{eq:approx_res} are given by
\begin{equation}
	\Psi_j'(d_j^2) = \frac{5 \eta_j}{d_j^2 \log 10}.
\end{equation}

Note that these weights are undefined when $d_j = 0$. In practice, this is not a problem as we can clamp $d_j$ to some suitable small number, e.g., $d_j \gets \max(d_j, 10^{-3})$ for $j = 1,\dotsc,m$.

\subsection{Trilateration With Partially Known Receiver Position}
\label{sec:partially_known_receiver}
In some scenarios, the receiver position might be partially known, e.g., when noiseless Angle Of Arrival (AOA) measurements are present. In 3D, a known azimuth angle could confine the receiver to a vertical plane, while an additional elevation angle could further confine it to a line. With a suitable transformation, the restricted space can be aligned to the axes resulting in partially known receiver coordinates.
It turns out that, modifying the cost function in \eqref{eq:approx_cost} to allow for this, reduces the trilateration problem to another one in a lower dimension.

Assume that $k$ coordinates of the receiver position are known. We can then partition $\vec{x}$ into $\vec{x}' \in \RR^{n-k}$ and $\vec{x}'' \in \RR^k$, representing the unknown and known coordinates of $\vec{x}$, respectively. Correspondingly, we partition the sender positions $\vec{s}_j$ into $\vec{s}_j' \in \RR^{n-k}$ and $\vec{s}_j'' \in \RR^k$. Given that $\|\vec{x}-\vec{s}_j\|^2 = \|\vec{x}'-\vec{s}_j'\|^2 + \|\vec{x}''-\vec{s}_j''\|^2$, we can rewrite the approximate residuals in \eqref{eq:approx_res} as
\begin{equation}
	\tilde{r}_j(\vec{x}) = \Psi_j'(d_j^2) \left( \|\vec{x}'-\vec{s}_j'\|^2 - (d_j')^2 \right).
\end{equation}
where $(d_j')^2 = d_j^2 - \|\vec{x}''-\vec{s}_j''\|^2$. These residuals are on the same form as \eqref{eq:approx_res} but over a lower dimension and can thus be solved using the proposed method.

\section{Eigenvalue Formulation}
\label{sec:eigenvalue_formulation}
In this section, we will derive a method for minimizing $h(\vec{x})$ in \eqref{eq:approx_cost} by transforming the first-order optimality conditions into an eigenvalue problem. We then prove that the global minimizer can be extracted from the largest real eigenvalue. We will also show how to handle degenerate cases, where more than one global minimizer exists.

Differentiating $h(\vec{x})$ and collecting the terms by degree yields
\begin{align}
	\nabla h(\vec{x}) &= \sum_{i=1}^m \sum_{j=1}^m w_{ij} (\|\vec{x}-\vec{s}_i\|^2 - d_i^2)(\vec{x}-\vec{s}_j) \label{eq:gradient} \\
	&= \left(\sum_{i=1}^m \sum_{j=1}^m w_{ij} \right) (\vec{x}^T\vec{x})\vec{x} \\
	&- ((\vec{x}^T\vec{x}) \mat{I} + 2\vec{x}\vec{x}^T)\left(\sum_{i=1}^m \sum_{j=1}^m w_{ij}\vec{s}_i\right) \\
	&+ \left(\sum_{i=1}^m \sum_{j=1}^m w_{ij} \left(2\vec{s}_j\vec{s}_i^T + (\vec{s}_i^T\vec{s}_i-d_i^2) \mat{I} \right) \right)\vec{x} \label{eq:gradient_A} \\
	&- \sum_{i=1}^m \sum_{j=1}^m w_{ij}(\vec{s}_i^T\vec{s}_i - d_i^2)\vec{s}_j, \label{eq:gradient_g}
\end{align}
where in \eqref{eq:gradient} we exploited the symmetry of $\mat{W}$.

To solve $\nabla h(\vec{x}) = 0$, we will simplify the gradient expression in three ways.
First, note that the cost function is homogeneous in $w_{ij}$, and since $\mat{W} \succ 0$, we can w.l.o.g. assume that $\sum_{ij} w_{ij} = 1$. This removes the coefficient of the third-order term.
Second, similar to \cite{zhou_closed-form_2011}, by applying the translation $\vec{t} = \sum_{ij} w_{ij}\vec{s}_i$ to the senders, i.e., $\vec{s}_j \leftarrow \vec{s}_j - \vec{t}$, we ensure that $\sum_{ij} w_{ij}\vec{s}_i = 0$, canceling the second-order term. This does not change the problem as long as any solution $\vec{x}$ is translated back accordingly. The gradient can now be written as
\begin{equation}
	\label{eq:diff_Ag}
	\nabla h(\vec{x}) = (\vec{x}^T\vec{x})\vec{x} - \mat{A}\vec{x} + \vec{g},
\end{equation}
where $\mat{A}$ and $\vec{g}$ are given in \eqref{eq:gradient_A} and \eqref{eq:gradient_g}, respectively.
Third, note that $\mat{A}$ is real and symmetric and can thus be diagonalized by an orthogonal matrix $\mat{Q}$.
Letting $\vec{y} = \mat{Q}^T\vec{x}$, we construct the new cost function $f(\vec{y}) = h(\mat{Q}\vec{y})$. Note that, minimizing $f(\vec{y})$ is equivalent to minimizing $h(\vec{x})$, where the senders $\vec{s}_j$ have been rotated using $\mat{Q}$.
The gradient of $f(\vec{y})$ now becomes
\begin{equation}
	\label{eq:diff_Db}
	\nabla f(\vec{y}) = (\vec{y}^T\vec{y})\vec{y} - \mat{D}\vec{y} + \vec{b},
\end{equation}
where $\mat{D} = \mat{Q}^T\mat{A}\mat{Q}$ is diagonal and $\vec{b} = \mat{Q}^T\vec{g}$. Furthermore, let the elements of $\mat{D}$ be sorted such that $D_{11} \geq D_{22} \geq \dotsb \geq D_{nn}$.

Solving for the stationary points is equivalent to solving the $n$ equations
\begin{equation}
	\label{eq:first_order_opt}
	(\vec{y}^T\vec{y})y_k - D_{kk}y_k + b_k = 0,
\end{equation}
where $k=1,\dotsc,n$.
Multiplying the $k$th equation in \eqref{eq:first_order_opt} with $y_k$ we get
\begin{equation}
	\label{eq:first_order_opt2}
	(\vec{y}^T\vec{y})y_k^2 - D_{kk}y_k^2 + b_k y_k = 0,
\end{equation}
which, naturally, are also satisfied at a stationary point.
Letting $\vec{y}^2 = (y_1^2, y_2^2, \dotsc, y_n^2)^T$ denote the vector of squared coordinates, we use \eqref{eq:first_order_opt2}, \eqref{eq:first_order_opt} and the trivial equation $\vec{y}^T\vec{y} = \vec{1}^T\vec{y}^2$ to form the eigendecomposition
\begin{equation}
	\label{eq:eig_M}
	(\vec{y}^T\vec{y})
	\begin{pmatrix}
		\vec{y}^2 \\
		\vec{y} \\
		1
	\end{pmatrix}
	=
	\underbrace{
	\begin{pmatrix}
		\mat{D}   & -\diag(\vec{b}) & \vec{0}  \\
		\mat{O}   & \mat{D}         & -\vec{b} \\
		\vec{1}^T & \vec{0}^T       & 0        \\
	\end{pmatrix}
	}_{\triangleq \mat{M}}
	\begin{pmatrix}
		\vec{y}^2 \\
		\vec{y} \\
		1
	\end{pmatrix},
\end{equation}
where $\mat{O}$ denotes the zero matrix, and $\vec{1}$ denotes a vector of ones.
Note that the matrix $\mat{M}$ does not depend on $\vec{y}$, and any stationary point $\vec{y}$ of $f(\vec{y})$ corresponds to an eigenpair of $\mat{M}$ on the form in \eqref{eq:eig_M}.
In the following section, we will investigate the converse, i.e., how to extract all stationary points given the eigenvalues of $\mat{M}$.

\subsection{Finding All Stationary Points}
\label{sec:stationary_points}

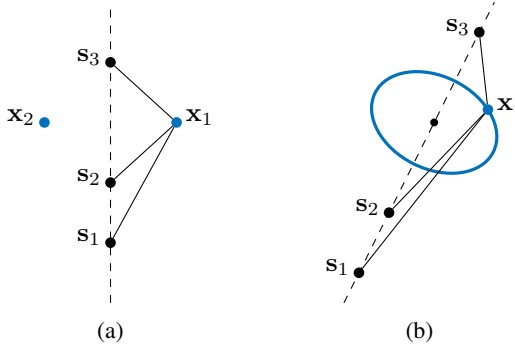
\begin{figure}
	\centering
	\begin{subfigure}{0.45\columnwidth}
		\centering
		\begin{tikzpicture}[
			dot/.style = {circle, fill, minimum size=#1,
				inner sep=0pt, outer sep=0pt},
			dot/.default = 4pt  
			]
			\draw[dashed] (0,-2) -- (0,2)
				node [pos=0.6] (c) {}
				node [pos=0.2, dot, label={[left]$\vec{s}_1$}] (s1) {}
				node [pos=0.4, dot, label={[left]$\vec{s}_2$}] (s2) {}
				node [pos=0.8, dot, label={[left]$\vec{s}_3$}] (s3) {};
			\node[circle, minimum size=50, inner sep=0pt, outer sep=0pt] (circ) at (c) {};
			\node[dot, label={[right]$\vec{x}_1$}, color=accentcol] (x) at (circ.0) {};
			\node[dot, label={[left]$\vec{x}_2$}, color=accentcol] (x2) at (circ.180) {};
			\draw (s1) -- (x);
			\draw (s2) -- (x);
			\draw (s3) -- (x);
		\end{tikzpicture}
		\caption{}
		\label{fig:degenerate_2d}
	\end{subfigure}
	\begin{subfigure}{0.45\columnwidth}
	\centering
	\begin{tikzpicture}[
		dot/.style = {circle, fill, minimum size=#1,
			inner sep=0pt, outer sep=0pt},
		dot/.default = 4pt  
		]
		\draw[dashed] (-1,-2) -- (1,2)
		node [pos=0.6, dot=3pt] (c) {}
		node [pos=0.1, dot, label={[left]$\vec{s}_1$}] (s1) {}
		node [pos=0.3, dot, label={[left]$\vec{s}_2$}] (s2) {}
		node [pos=0.9, dot, label={[left]$\vec{s}_3$}] (s3) {};
		\node[draw, ellipse, very thick, color=accentcol, minimum width=35, minimum height=50, inner sep=0pt, outer sep=0pt, rotate=63.43] (circ) at (c) {};
		\node[dot, label={[right]$\vec{x}$}, color=accentcol] (x) at (circ.-50) {};
		\draw (s3) -- (x);
		\draw (s1) -- (x);
		\draw (s2) -- (x);
	\end{tikzpicture}
	\caption{}
	\label{fig:degenerate_3d}
	\end{subfigure}
	\caption{Two degenerate sender configurations.
		(a) When the senders are collinear in 2D, there are two possible solutions, $\vec{x}_1$ and $\vec{x}_2$. (b) When the senders are collinear in 3D, there are infinitely many solutions located on a circle.}
	\label{fig:degenerate}
\end{figure}

It is not necessarily the case that any eigenpair of $\mat{M}$ is on the form in \eqref{eq:eig_M}. 
However, to find all stationary points it is sufficient to, for each eigenvalue, find all corresponding eigenvectors on the form in \eqref{eq:eig_M}. The approach for doing this can be divided into two cases, depending on whether $\lambda \mat{I} - \mat{D}$ is singular.

\begin{proposition}
	\label{prop:stationary_point}
	If $\lambda$ is an eigenvalue of $\mat{M}$ and $\lambda \mat{I} - \mat{D}$ has full rank, then $\vec{y} = -(\lambda \mat{I} - \mat{D})^{-1} \vec{b}$ is the unique stationary point of $f(\vec{y})$ satisfying $\lambda = \vec{y}^T\vec{y}$.
\end{proposition}
\begin{proof}
	Let $\vec{v} = (\vec{v}_1^T, \vec{v}_2^T, v_3)^T$ be an eigenvector corresponding to $\lambda$. If $v_3 = 0$, then, from the definition of $\mat{M}$, $(\lambda \mat{I} - \mat{D})\vec{v}_2 = 0 \Rightarrow \vec{v}_2 = 0 \Rightarrow (\lambda \mat{I} - \mat{D})\vec{v}_1 = 0 \Rightarrow \vec{v}_1 = 0 \Rightarrow \vec{v} = 0$, a contradiction. Consequently, we can w.l.o.g. assume $v_3 = 1$. Then $\vec{v}_2 = -(\lambda \mat{I} - \mat{D})^{-1} \vec{b}$ and $\vec{v}_1 = -(\lambda \mat{I} - \mat{D})^{-1}\diag(\vec{b})\vec{v}_2 = \vec{v}_2^2$. From the last row of $\mat{M}$ we get $\lambda = \vec{v}_2^T\vec{v}_2$, and  the eigenpair $(\lambda, \vec{v})$ is on the form in \eqref{eq:eig_M} with $\vec{y} = \vec{v}_2$. Consequently, $\vec{y}$ is a stationary point of $f(\vec{y})$ satisfying $\lambda = \vec{y}^T\vec{y}$. Uniqueness follows from the fact that $\vec{v}$ is the unique eigenvector corresponding to $\lambda$.
\end{proof}

The case when $\lambda \mat{I} - \mat{D}$ is singular we denote as a \emph{degenerate case}. As we will see, this corresponds to when the trilateration problem is underdefined and has multiple, possibly infinitely many, solutions.
Let $\mat{A}^+$ denote the Moore-Penrose pseudo inverse of $\mat{A}$.

\begin{proposition}
	\label{prop:stationary_point_singular}
	If $\lambda$ is an eigenvalue of $\mat{M}$, $\lambda \mat{I} - \mat{D}$ is singular, and $(\lambda \mat{I} - \mat{D})\vec{y} = -\vec{b}$ has a solution, then $\vec{y} = \vec{y}_p + \vec{y}_h$ is a stationary point, where $\vec{y}_p = -(\lambda \mat{I} - \mat{D})^+\vec{b}$ and $\vec{y}_h \in \ker(\lambda \mat{I} - \mat{D})$ such that $\vec{y}_h^T\vec{y}_h = \lambda - \vec{y}_p^T\vec{y}_p$.
\end{proposition}
\begin{proof}
    Due to the properties of the pseudo inverse $\vec{y}_p^T\vec{y}_h = 0$, and the last constraint on $\vec{y}_h$ then ensures that $\lambda = \vec{y}^T\vec{y}$.
    The fact that $\vec{y}$ satisfies $(\lambda \mat{I} - \mat{D})\vec{y} = -\vec{b}$ is then equivalent to $\nabla f(\vec{y}) = 0$, and $\vec{y}$ is a stationary point.
\end{proof}

The geometric interpretation of Proposition~\ref{prop:stationary_point_singular} is that, provided $\lambda$ is real and $\lambda - \vec{y}_p^T\vec{y}_p \geq 0$, the stationary points satisfying $\lambda = \vec{y}^T\vec{y}$ consists of a hypersphere centered at $\vec{y}_p$ with radius $\sqrt{\lambda - \vec{y}_p^T\vec{y}_p}$.
Although embedded in $\RR^n$, the dimension of the hypersphere is given by $n - 1 - \rank(\lambda \mat{I} - \mat{D})$.
In particular, when $\lambda \mat{I} - \mat{D}$ has rank $n-1$ there is still a finite number of solutions (two).
Strictly speaking, there is also a finite number of solutions (one) when $\lambda = \vec{y}_p^T\vec{y}_p$ regardless of the rank.

\figurename~\ref{fig:degenerate} shows two degenerate cases resulting from the senders not spanning the space. These types of sender configurations will result in multiple global minima for any cost being a function of $\| \vec{x} - \vec{s}_j \|$, e.g., \eqref{eq:ls}, \eqref{eq:sls}, and $h_0(\vec{x})$, unless a global minimum exists in the affine subspace spanned by the senders.
In \figurename~\ref{fig:degenerate_3d}, we have $\rank(\lambda \mat{I} - \mat{D}) = n-2 = 1$, and this rank persists if $\vec{x}$ becomes collinear with the senders, i.e., we have the scenario described above where $\lambda = \vec{y}_p^T\vec{y}_p$, and a single unique solution exists.

A different type of degenerate case is shown in \figurename~\ref{fig:degenerate_2d_2}, where the senders indeed span the space. Here, \eqref{eq:ls} has a finite number of solutions while \eqref{eq:sls} has infinitely many. Given the slightly shorter distance measurements $d_j=1.5$ for $j = 1, \dotsc, 4$, \eqref{eq:ls} has a unique solution at the origin, while the solutions to \eqref{eq:sls} still consist of a circle, this time with radius $0.5$. Consequently, there is no guarantee that the approximate cost $h(\vec{x})$ and the original cost $h_0(\vec{x})$ have the same number of solutions or even both finitely many.

Note that Propositions~\ref{prop:stationary_point} and \ref{prop:stationary_point_singular} hold for complex eigenvalues and complex solutions to $\nabla f(\vec{y}) = \vec{0}$.
However, in this work, we are only interested in the real solutions and thus only need to consider the real eigenvalues of $\mat{M}$.

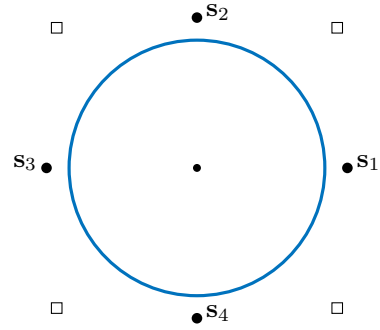
\begin{figure}
	\centering
	\begin{tikzpicture}[
		dot/.style = {circle, fill, minimum size=#1,
			inner sep=0pt, outer sep=0pt},
		dot/.default = 4pt,  
        square/.style = {draw, minimum size=#1,
			inner sep=0pt, outer sep=0pt},
		square/.default = 4pt,
		]
        
		\node[draw, circle, very thick, color=accentcol, minimum size=3.4cm, inner sep=0pt, outer sep=0pt] at (0,0) {};
		\node[dot=3pt] at (0,0) {};
		\node[dot, label={[right]$\vec{s}_1$}] (s1) at (2,0) {};
		\node[dot, label={[right]$\vec{s}_2$}] (s2) at (0,2) {};
		\node[dot, label={[left]$\vec{s}_3$}] (s3) at (-2,0) {};
		\node[dot, label={[right]$\vec{s}_4$}] (s4) at (0,-2) {};
		\node[square] at (1.8646,1.8646) {};
		\node[square] at (1.8646,-1.8646) {};
		\node[square] at (-1.8646,1.8646) {};
		\node[square] at (-1.8646,-1.8646) {};
	\end{tikzpicture}
	\caption{Degenerate case in the plane occurring when the senders are evenly distributed on the unit circle and the distance measurements are $d_j = 1.65$ for $j=1,\dotsc,4$. The blue circle with radius $0.85$ indicates the solutions to \eqref{eq:sls}, while the four squares indicate the solutions to \eqref{eq:ls}.
	}
	\label{fig:degenerate_2d_2}
\end{figure}

\subsection{Finding the Global Minimizer}
\label{sec:global_minimizer}
The global minimizer of $f(\vec{y})$ can be found by enumerating all stationary points and evaluating the cost. However, this is unnecessary, as the following two propositions show that the global minimum corresponds to the largest real eigenvalue of $\mat{M}$.

\begin{proposition}
	\label{prop:global_min}
	The point $\vec{y}$ is a global minimizer of $f(\vec{y})$ if and only if $\nabla f(\vec{y}) = 0$ and $\vec{y}^T\vec{y} \geq D_{11}$. If $\vec{y}^T\vec{y} > D_{11}$, then $\vec{y}$ is the unique global minimizer of $f(\vec{y})$.
\end{proposition}
\begin{proof}
	Minimizing $f(\vec{y})$ is equivalent to the generalized trust region subproblem
	\begin{align}
		\minimize_{\vec{y},\lambda} \quad & \frac{1}{4} \lambda^2 - \frac{1}{2} \vec{y}^T\mat{D}\vec{y}+\vec{b}^T\vec{y}, \\
		\text{subject to} \quad &  \frac{1}{2} (\vec{y}^T\vec{y}-\lambda) = 0, \notag
	\end{align}
	where the Lagrangian is given by
	\begin{equation}
		\mathcal{L}(\vec{y},\lambda,\nu) = \frac{1}{4} \lambda^2 - \frac{1}{2} \vec{y}^T\mat{D}\vec{y}+\vec{b}^T\vec{y} + \frac{1}{2} \nu (\vec{y}^T\vec{y}-\lambda),
	\end{equation}
	and $\nu$ is the Lagrange multiplier.
	By \cite[Theorem 3.2]{more_generalizations_1993}, $(\vec{y},\lambda)$ is a global minimizer if and only if it for some $\nu$ satisfies
	\begin{align}
		\nabla_\vec{y}\mathcal{L}(\vec{y},\lambda,\nu) &= -\mat{D}\vec{y}+\vec{b} + \nu \vec{y} = 0, \\
		\nabla_\lambda\mathcal{L}(\vec{y},\lambda,\nu) &= \frac{1}{2}(\lambda - \nu) = 0, \\
		\nabla_{\vec{y},\lambda}^2 \mathcal{L}(\vec{y},\lambda,\nu) &\succeq 0
		\quad \Leftrightarrow \quad \nu\mat{I}-\mat{D} \succeq 0,
	\end{align}
	and $\lambda = \vec{y}^T\vec{y}$, which is equivalent to $\nabla f(\vec{y}) = 0$ and $\lambda = \vec{y}^T\vec{y} \geq D_{11}$.
	Uniqueness of the global minimizer is given by \cite[Theorem 4.1]{more_generalizations_1993} when $\nabla_{\vec{y},\lambda}^2 \mathcal{L}(\vec{y},\lambda,\nu) \succ 0 \Leftrightarrow \vec{y}^T\vec{y} > D_{11}$.
\end{proof}

\begin{proposition}
	\label{prop:largest_real}
	For any global minimizer $\vec{y}$ of $f(\vec{y})$, we have $\lambda_{\text{max}} = \vec{y}^T\vec{y}$, where $\lambda_{\text{max}}$ is the largest real eigenvalue of $\mat{M}$.
\end{proposition}
\begin{proof}
   	From the definition of $f(\vec{y})$, a global minimizer $\vec{y}$ must exist,
    and from Proposition III.3, $\vec{y}^T\vec{y} \geq D_{11}$. Since $\vec{y}^T\vec{y}$ is an eigenvalue of $\mat{M}$, we have $\lambda_{\text{max}} \geq D_{11}$.
	
	If $\lambda_{\text{max}} > D_{11}$, then $\vec{y} = -(\lambda_{\text{max}} \mat{I} - \mat{D})^{-1} \vec{b}$ is a stationary point satisfying $\lambda_{\text{max}} = \vec{y}^T\vec{y}$ by Proposition~\ref{prop:stationary_point} and a unique global minimizer by Proposition~\ref{prop:global_min}.
	
	If $\lambda_{\text{max}} = D_{11}$, then clearly any global minimizer $\vec{y}$ satisfies $\lambda_{\text{max}} = \vec{y}^T\vec{y}$, or there would exist a real eigenvalue larger than $\lambda_{\text{max}}$.
\end{proof}

When solving trilateration problems, we are often working in 2D or 3D space, and $\mat{M}$ will have size $5 \times 5$ or $7 \times 7$, respectively. For such small problems, it is computationally cheap to calculate all eigenvalues, even though Proposition~\ref{prop:largest_real} has shown that we are only interested in the largest real eigenvalue. Nevertheless, there are efficient algorithms for finding the eigenvalue with the largest real part, i.e., the rightmost eigenvalue \cite{arnoldi_principle_1951, lehoucq_arpack_1998}, but for these to be applicable, we first need to show that the largest real eigenvalue also is the rightmost one.

\begin{proposition}[{cf. \cite{adachi_solving_2017}[Theorem 3.4]}]
	\label{prop:rightmost}
	The rightmost eigenvalue of $\mat{M}$ is real.
\end{proposition}
\begin{proof}
	Assume for contradiction that $\lambda = \alpha + \beta i$ is the rightmost eigenvalue of $\mat{M}$, where $\alpha, \beta \in \RR$, $\alpha \geq D_{11}$, $\beta \neq 0$ and $i$ is the imaginary unit. Since $\lambda = \alpha - \beta i$ also is an eigenvalue, we can assume $\beta > 0$. Then $\lambda I - \mat{D}$ has full rank and by Proposition~\ref{prop:stationary_point}
	\begin{equation}
		\lambda - \sum_{k=1}^n \frac{b_k^2}{(\lambda-D_{kk})^2} = \lambda - \vec{y}^T\vec{y} = 0.
	\end{equation}
	However, $\Im(b_k^2/(\lambda-D_{kk})^2) \leq 0$ for $k = 1, \dotsc, n$, implying $\Im(\lambda - \vec{y}^T\vec{y}) > 0$. This is a contradiction, and $\lambda$ cannot be an eigenvalue of $\mat{M}$.
\end{proof}

\subsection{Alternative Eigendecomposition}
\label{sec:alt_eig}
In the previous sections, we have worked with the matrix $\mat{D}$ found by diagonalizing $\mat{A}$. This diagonalization step can be mitigated by instead of $\mat{M}$ considering the eigenvalues of
\begin{equation}
	\label{eq:eig_Ma}
	\mat{M}_A =
	\begin{pmatrix}
		\mat{A} &  \mat{I} & \vec{0} \\
		\mat{O} & \mat{A} & -\vec{g} \\
		-\vec{g}^T & \vec{0}^T & 0 \\
	\end{pmatrix}.
\end{equation}
This matrix is similar to $\mat{M}_D = \mat{P}^{-1}\mat{M}_A\mat{P}$, where
\begin{equation}
	\label{eq:eig_Md}
	\mat{M}_D =
	\begin{pmatrix}
		\mat{D} &  \mat{I} & \vec{0} \\
		\mat{O} & \mat{D} & -\vec{b} \\
		-\vec{b}^T & \vec{0}^T & 0 \\
	\end{pmatrix}
    ,\quad
    \mat{P} =
	\begin{pmatrix}
		\mat{Q} & & \\
		& \mat{Q} & \\
		& & 1
	\end{pmatrix}.
\end{equation}
Furthermore, using, e.g., Laplace expansion, we can show that $\mat{M}$ and $\mat{M}_D$ have the same characteristic polynomial. Consequently, $\mat{M}$, $\mat{M}_D$, and $\mat{M}_A$ all share the same eigenvalues.
Because of this, the propositions from the previous sections now naturally transfer to this new formulation using $\mat{M}_A$.
In particular, if $\lambda_{\text{max}}$ is the largest real eigenvalue of $\mat{M}_A$ and $\lambda_{\text{max}} \mat{I} - \mat{A}$ has full rank, then the global minimizer of $h(\vec{x})$ is given by $\vec{x} = -(\lambda_{\text{max}} \mat{I} - \mat{A})^{-1} \vec{g}$.

It is worth noting that $\mat{M}_A$ closely resembles the matrix constructed in \cite{adachi_solving_2017} for solving the Trust-Region Subproblem (TRS). However, the problem considered here does not belong to that class of problems, and consequently, their method is not directly applicable.
In a later work \cite{adachi_eigenvalue-based_2019}, the method proposed in \cite{adachi_solving_2017} was extended to the GTRS. However, applying their method would require the introduction of an additional unknown corresponding to $\alpha$ in \eqref{eq:sls_constrained} and results in a $(2n+3) \times (2n+3)$ generalized eigenvalue problem, as opposed to the here proposed $(2n+1) \times (2n+1)$ ordinary eigenvalue problem.

\subsection{The Proposed Algorithm}
The results from the previous sections now admit an algorithm for finding the global minimizer of $h(\vec{x})$. The simplest version of the proposed method is shown in Algorithm~\ref{alg:proposed_simple}, but we will offer a few improvements resulting in the proposed method in Algorithm~\ref{alg:proposed}.

\begin{algorithm}[tb]
	\caption{Simplified Trilateration}
	\hspace*{\algorithmicindent} \textbf{Input:} Sender positions $\vec{s}_j$, distances $d_j$, weights $w_{ij}$ \\
	\hspace*{\algorithmicindent} \textbf{Output:} Receiver position $\vec{x}$ 
	\begin{algorithmic}[1]
		\State Normalize weights: $w_{ij} \leftarrow w_{ij} / \sum_{ij} w_{ij}$.
		\State Translate senders: $\vec{s}_j \leftarrow \vec{s}_j - \vec{t}$, where $\vec{t} = \sum_{ij} w_{ij}\vec{s}_i$.
        \State Calculate $\mat{A}$ and $\vec{g}$ in \eqref{eq:diff_Ag}.
		\State Find largest real eigenvalue $\lambda_\text{max}$ of $\mat{M}_A$ \eqref{eq:eig_Ma}.
		\State Find receiver position as $\vec{x} = -(\lambda_\text{max} \mat{I} - \mat{A})^{-1} \vec{g}$.
		\State Undo translation: $\vec{x} \leftarrow \vec{x} + \vec{t}$.
	\end{algorithmic}
	\label{alg:proposed_simple}
\end{algorithm}

\begin{algorithm}[tb]
	\caption{Trilateration}
	\hspace*{\algorithmicindent} \textbf{Input:} Sender positions $\vec{s}_j$, distances $d_j$, weights $w_{ij}$ \\
	\hspace*{\algorithmicindent} \textbf{Output:} Receiver position(s) $\vec{x}$ 
	\begin{algorithmic}[1]
		\State Normalize weights: $w_{ij} \leftarrow w_{ij} / \sum_{ij} w_{ij}$.
		\State Translate senders: $\vec{s}_j \leftarrow \vec{s}_j - \vec{t}$, where $\vec{t} = \sum_{ij} w_{ij}\vec{s}_i$.
        \State Calculate $\mat{D}$ and $\vec{b}$ in \eqref{eq:diff_Db}.
		\State Find largest real eigenvalue $\lambda_\text{max}$ of $\mat{M}$ \eqref{eq:eig_M}.
		\If {$\rank(\lambda_\text{max} \mat{I} - \mat{D}) = n$}
			\State Solve for $\vec{y}$ using \eqref{eq:y_degen_1}-\eqref{eq:y_degen_2} and choose the sign in \eqref{eq:y_degen_2} such that $\sgn(y_1) = -\sgn(b_1)$.
		\ElsIf {$\rank(\lambda_\text{max} \mat{I} - \mat{D}) = n-1$}
			\State Solve for the two solutions $\vec{y}_1$ and $\vec{y}_2$ using \eqref{eq:y_degen_1}-\eqref{eq:y_degen_2}.
		\Else
			\State The problem is ill-defined. Return nothing.
		\EndIf
		\State Undo rotation: $\vec{x} = \mat{Q}\vec{y}$.
		\State Undo translation: $\vec{x} \leftarrow \vec{x} + \vec{t}$.
	\end{algorithmic}
	\label{alg:proposed}
\end{algorithm}

While the formulation using $\mat{M}_A$ in Section~\ref{sec:alt_eig} avoids the eigendecomposition of $\mat{A}$, explicitly calculating $\mat{D}$ simplifies calculations related to the degenerate cases, e.g., it is trivial to find the rank, kernel, and Moore-Penrose pseudo inverse of $\lambda \mat{I} - \mat{D}$. Furthermore, we found that in our implementation the eigendecomposition of $\mat{M}$ is faster to calculate than that of $\mat{M}_A$, resulting in an overall faster solver.

Another issue with Algorithm~\ref{alg:proposed_simple} is that it does not handle the degenerate cases. When $\rank(\lambda_\text{max} \mat{I} - \mat{D}) < n-1$, there is almost always an infinite number of solutions, and the problem is, for our purposes, ill-defined.
However, when $\rank(\lambda_\text{max} \mat{I} - \mat{D}) = n-1$, there are two solutions given by (see Proposition~\ref{prop:stationary_point_singular})
\begin{align}
	\label{eq:y_degen_1}
	y_k &= - \frac{b_k}{\lambda_\text{max} - D_{kk}} \quad \text{for} \quad k = 2, \dotsc, n, \\
	\label{eq:y_degen_2}
	y_1 &= \pm \sqrt{\lambda_\text{max} - \sum_{k=2}^n y_k^2}.
\end{align}
If $\lambda_\text{max} \mat{I} - \mat{D}$ has full rank, a single solution is given by Proposition~\ref{prop:stationary_point}. However, when approaching a degenerate case the numerical stability of $\vec{y} = -(\lambda_\text{max} \mat{I} - \mat{D})^{-1} \vec{b}$ gets worse. To avoid this, we always treat the problem as if the rank is $n-1$ and use \eqref{eq:y_degen_1}-\eqref{eq:y_degen_2} also in the nondegenerate case. Only one of the two solutions is correct though, so we choose the sign in \eqref{eq:y_degen_2} such that $\sgn(y_1) = -\sgn(b_1)$. This approach yields better numerical stability when transitioning to and from degenerate cases.
The improved proposed method is summarized in Algorithm~\ref{alg:proposed}.
Note that the rank check on $\lambda_\text{max} \mat{I} - \mat{D}$ is only used to determine whether one, two, or no solutions should be returned.
One could easily modify the algorithm to never fail and always return one of the possibly many global minimizers.

\section{Experiments with Synthetic Data}
\label{sec:experiment_synth}
In this section, we compare the proposed methods in Algorithms~\ref{alg:proposed_simple} and \ref{alg:proposed} with several other methods from the literature \cite{zhou_closed-form_2011, beck_exact_2008, beck_iterative_2008, luke_simple_2017, ismailova_penalty_2016}. In addition, we include a linear approach which is the unconstrained version of \eqref{eq:sls_constrained}, and one method solving \eqref{eq:sls_constrained} using the generalized eigenvalue decomposition presented in \cite{adachi_eigenvalue-based_2019}. To solve the SDP in \cite{beck_exact_2008} we used Hypatia \cite{coey_solving_2022}, and for the CCP in \cite{ismailova_penalty_2016} we used Clarabel \cite{clarabel} initialized with the mean sender position. All methods have been implemented in Julia \cite{bezanzon_julia_2017}, and we share our code at the end of the paper.

It should be noted that some of these methods minimize the ML cost \eqref{eq:ls} and some minimize $h(\vec{x})$ with the weights $w_{jj} = 1 / 4d_j^2$, $w_{ij} = 0$ for $i \neq j$, as proposed in Section~\ref{sec:noise_distributions}. However, Zhou \cite{zhou_closed-form_2011} does not trivially allow for a similar weighting and is minimizing the unweighted cost \eqref{eq:sls}.

\subsection{Execution Time}
Comparing the execution time of the different methods is not trivial as it depends on several factors such as termination criteria, initialization, noise levels, number of senders, programming language, and implementation choices not specified in the papers. We have attempted to make the comparison fair by adjusting termination criteria where possible to produce errors in the order of $10^{-6}$. The proposed method has no such options and will yield errors in the order of $10^{-15}$.

To perform the tests, synthetic data were generated consisting of a single receiver $\vec{x} \in \RR^3$ and $m$ senders $\vec{s}_j \in \RR^3$ with coordinates drawn from the standard normal distribution $\mathcal{N}(0,1)$. The distance measurements were calculated without any added noise. A benchmark was set up where each solver was run \num{10000} times or for a maximum total of 10 seconds, each time with different synthetic data. The benchmark was run on an AMD Ryzen Threadripper 3990X, and the resulting median run times are shown in Table~\ref{tab:execution_time} for $m = 4, 10, 100$.

\sisetup{
	round-mode = figures,
	round-precision = 2,
}

\begin{table}
	\centering
	\caption{Execution time for several trilateration methods from the literature.}
	\label{tab:execution_time}
	\begin{tabular}{lrrr}
		\toprule
		Method & \multicolumn{3}{c}{Execution time} \\
		  & $m=4$ & $m=10$ & $m=100$ \\
		\midrule
Zhou \cite{zhou_closed-form_2011} & \SI{6.1}{\micro\second} & \SI{6.2}{\micro\second} & \SI{7.3}{\micro\second} \\
Linear & \SI{1.4}{\micro\second} & \SI{1.7}{\micro\second} & \SI{5.0}{\micro\second} \\
Beck SDR \cite{beck_exact_2008} & \SI{5.1212}{\milli\second} & \SI{14.5433}{\milli\second} & \SI{127.7241680}{\second} \\
Ismailova \cite{ismailova_penalty_2016} & \SI{13.6984}{\milli\second} & \SI{12.3945}{\milli\second} & \SI{35.7594}{\milli\second} \\
Adachi \cite{adachi_eigenvalue-based_2019} & \SI{36.7}{\micro\second} & \SI{39.5}{\micro\second} & \SI{42.4}{\micro\second} \\
Luke \cite{luke_simple_2017} & \SI{606.0}{\micro\second} & \SI{295.4}{\micro\second} & \SI{780.5}{\micro\second} \\
Beck SFP \cite{beck_iterative_2008} & \SI{571.9}{\micro\second} & \SI{384.0}{\micro\second} & \SI{2.7393}{\milli\second} \\
Beck SR-LS \cite{beck_exact_2008}& \SI{21.3}{\micro\second} & \SI{25.9}{\micro\second} & \SI{33.6}{\micro\second} \\
Proposed (Alg. 1) & \SI{30.8}{\micro\second} & \SI{31.8}{\micro\second} & \SI{35.3}{\micro\second} \\
Proposed (Alg. 2) & \SI{18.7}{\micro\second} & \SI{19.3}{\micro\second} & \SI{22.6}{\micro\second} \\
		\bottomrule
	\end{tabular}
\end{table}

\begin{figure*}
	\centering
	\includegraphics[width=\textwidth]{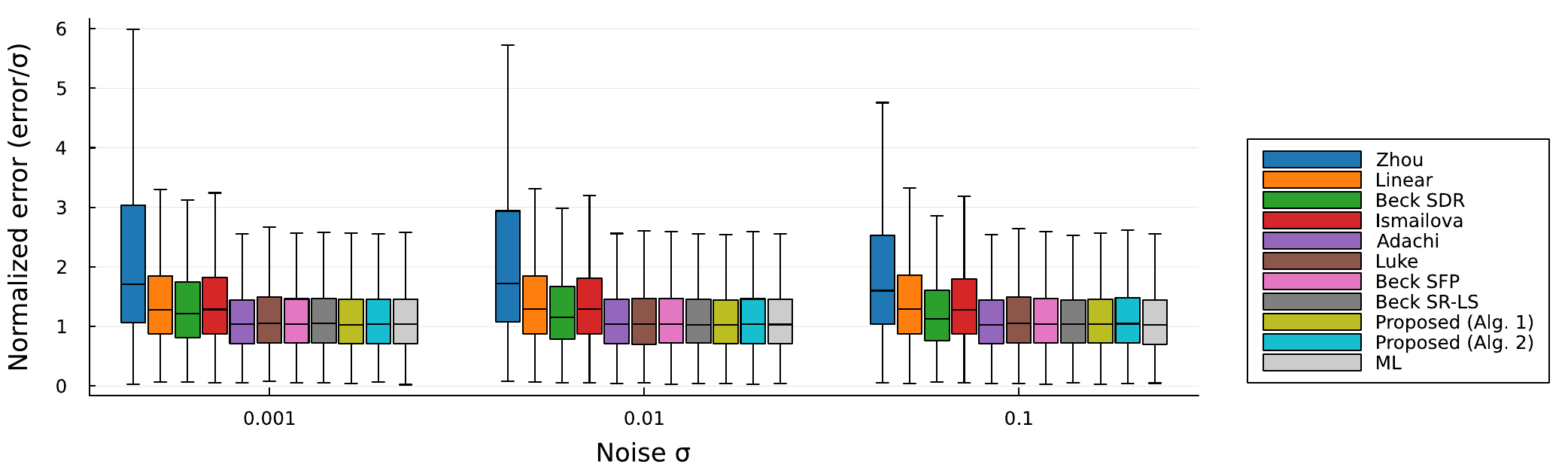}
	\caption{Normalized errors in estimated receiver position for several trilateration methods over various amounts of noise. The whiskers in the boxplot indicate the 1.5 IQR value and outliers are not shown.}
	\label{fig:synth_test}
\end{figure*}

As can be seen, the proposed method is competitive concerning execution time, being faster than all but the linear one and Zhou~\cite{zhou_closed-form_2011}. Note also, as mentioned before, that Algorithm~\ref{alg:proposed_simple} is not faster than Algorithm~\ref{alg:proposed} even though it avoids the eigendecomposition of $\mat{A}$. Luke \cite{luke_simple_2017} and Beck SFP \cite{beck_iterative_2008} are slower when $m=4$ compared to $m=10$. A possible explanation is that these iterative methods converge slower for near-degenerate cases, which are more likely to occur when $m=4$. They also become slower when $m=100$, which is expected as they use all sender positions in their inner iterations. In contrast to this, the proposed method quickly reduces the data to $\mat{A} \in \RR^{n \times n}$ and $\vec{g} \in \RR^n$, where the sizes only depend on the spatial dimension $n$ and not the number of senders $m$. Consequently, the proposed method maintains an almost constant execution time over the number of senders.
One way to speed up the iterative approaches is to initialize them with the linear solution. However, even for moderate noise, this initialization becomes sufficiently inaccurate for the solvers to remain slower than the proposed.
The convex relaxation methods, Beck SDR \cite{beck_exact_2008} and Ismailova \cite{ismailova_penalty_2016}, are slower by several orders of magnitude. The former scales horribly, possibly due to the number of unknowns growing quadratically with the number of senders.

\subsection{Gaussian Noise}
To evaluate the proposed method over various amounts of noise, a large set of synthetic datasets was constructed. The positions for $m=10$ senders and a single receiver in 3D space were sampled from the standard normal distribution $\mathcal{N}(\mat{0},\mat{I})$, after which distance measurements were calculated and subsequently perturbed by zero-mean Gaussian noise with a standard deviation of $\sigma \in \{0.001, 0.01, 0.1\}$. A total of \num{10000} datasets were constructed per $\sigma$.

\figurename~\ref{fig:synth_test} shows the errors in the estimated receiver positions normalized with the noise level $\sigma$ for a range of trilateration methods.
The maximum likelihood (ML) estimate, found using local optimization of \eqref{eq:ls}, initialized at the ground truth receiver position, is included for reference.
As can be seen, most of the methods performed similarly. For example, the mean error of the proposed algorithms is within \SI[round-precision=1]{1}{\percent} of the ML mean error. This illustrates that with suitable weights, the cost function in \eqref{eq:sls} provides a good approximation of \eqref{eq:ls}.
Zhou \cite{zhou_closed-form_2011} and the linear method performed worse. The poor performance of the former can be attributed to the assumption of $\sum_j \|\vec{x}-\vec{s}_j\|^2 = \sum_j d_j^2$ made, which is generally not true in the presence of noise. Not using the weighting in Section~\ref{sec:noise_distributions} contributes as well, but not to the same extent. Similarly, the linear method performs poorly because the neglected constraint $\alpha = \vec{x}^T\vec{x}$ in \eqref{eq:sls_constrained} is generally not satisfied with noisy measurements.
However, these methods may still perform well in low-noise situations and remain attractive due to being closed-form and fast. Contrary to this, Beck SDR~\cite{beck_exact_2008} and Ismailova~\cite{ismailova_penalty_2016} are less accurate than other methods while also being significantly slower.

\subsection{Degenerate Configurations}
\label{sec:degenerate_configurations}
Some of the methods in the literature do not handle degenerate cases and become numerically unstable as we approach such a scenario. In situations similar to that in \figurename~\ref{fig:degenerate_2d}, there are two possible solutions to the trilateration problem. Algorithm~\ref{alg:proposed} and Zhou~\cite{zhou_closed-form_2011} return two solutions in this case, while the remaining methods, if successful, only return one.

To investigate the numerical stability of the methods close to a degenerate case, we generate $m=6$ senders $\vec{s}_j \in \RR^n$ for $j = 1,\dotsc,6$ and a single receiver $\vec{x} \in \RR^n$ with coordinates sampled from $\mathcal{N}(0,1)$. However, the x-coordinate of each sender position is multiplied with a scaling factor. When this scaling factor approaches zero the senders become coplanar and a degenerate case occurs. The distance measurements are calculated without any added noise.

\begin{figure*}
\centering
	\begin{subfigure}{0.4\textwidth}
    	\includegraphics[scale=0.5]{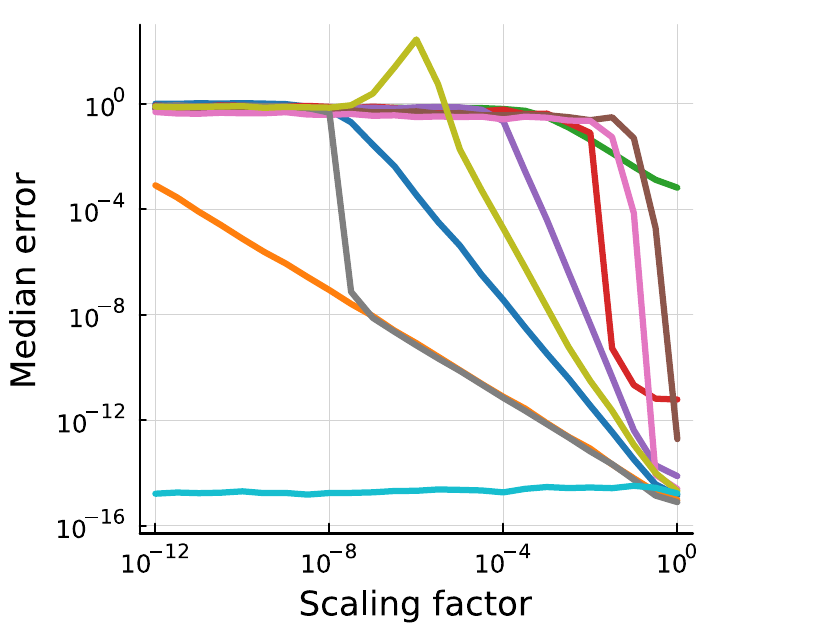}
    	\caption{}
    	\label{fig:degenerate_test_no_noise}
    \end{subfigure}
	\begin{subfigure}{0.4\textwidth}
    	\includegraphics[scale=0.5]{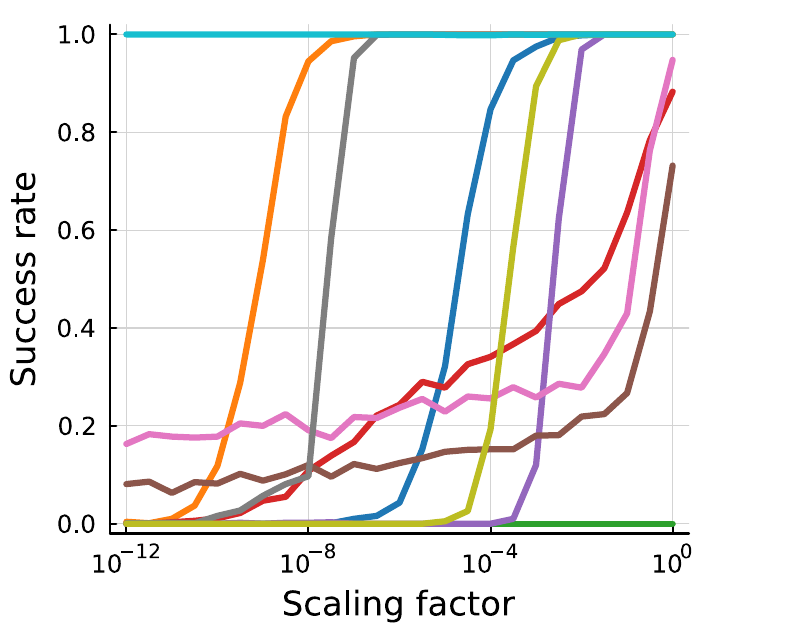}
    	\caption{}
    	\label{fig:degenerate_success_rate}
    \end{subfigure}
    \hspace{-5mm}
    \begin{subfigure}{0.15\textwidth}
    	\includegraphics[scale=0.5]{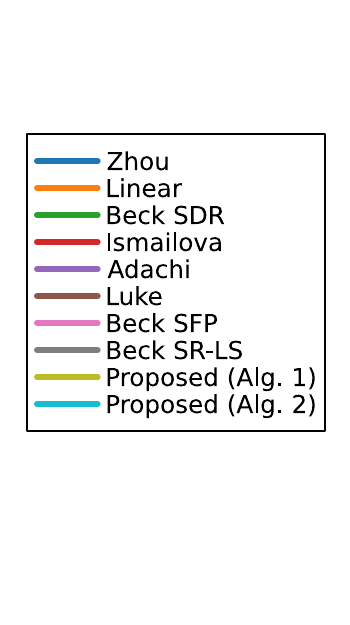}
    \end{subfigure}
    \caption{(a) Median error in estimated receiver position and (b) success rate over \num{1000} trials for a range of scaling factors. Success is defined as an error smaller than $10^{-6}$.
    The scaling factor is multiplied with the x-coordinate of each sender causing them to become coplanar when the factor approaches zero.
    }
    \label{fig:degenerate_test}
\end{figure*}

\figurename~\ref{fig:degenerate_test} shows the median error in the estimated receiver position and the success rate over \num{1000} trials for a range of scaling factors. Success is defined as an error smaller than $10^{-6}$, and if a method returns two solutions, the solution with the smallest error is used. As can be seen, Algorithm~\ref{alg:proposed} yields errors close to machine precision over the entire range of scaling factors. Due to the use of \eqref{eq:y_degen_1} and \eqref{eq:y_degen_2} also in the nondegenerate case, there is no obvious transition where the rank of $\lambda_\text{max} \mat{I} - \mat{D}$ changes.
All other methods performed notably worse.
In particular, Beck SR-LS~\cite{beck_exact_2008}, which is otherwise competitive concerning execution time and accuracy, gets progressively worse with smaller scaling factors. At a factor of $10^{-7}$, the so-called ``hard case'' of the GTRS occurs more often, which is not handled, and the method rapidly fails.
Adachi~\cite{adachi_eigenvalue-based_2019} is a general GTRS solver and should be able to handle the ``hard case''. However, the method is relatively complicated, so we omitted this in our implementation for simlicity. A complete implementation of the method should be expected to  accurately produce at least one solution.
This experiment also highlights a flaw in Zhou \cite{zhou_closed-form_2011}. In general, the method would perform similarly to Algorithm~\ref{alg:proposed} for this type of scenario. Unfortunately, it is sensitive specifically to the scaling in the x-coordinates done here.
Algorithm~\ref{alg:proposed_simple} performs poorly with a peak error of more than $10^2$ at the scaling factor $10^{-6}$. This peak is due to the eigensolver, and if $\mat{M}$ is used to find $\lambda_\text{max}$ instead of $\mat{M}_A$ the algorithm performs similarly to Zhou.
The iterative approaches Luke~\cite{luke_simple_2017}, Beck~SFP~\cite{beck_iterative_2008}, and Ismailova~\cite{ismailova_penalty_2016} 
would perform better with a different initialization, particularly one that produces two starting points corresponding to the two local minima of the cost function in this scenario. The proposed method does not require initialization and, consequently, no such design decisions are necessary.

\section{Experiments with Real Data}
\label{sec:experiment_real}

To demonstrate the practicality of the proposed algorithm, we will present a few results using RSS and Round-Trip Time (RTT) measurements from ten Wi-Fi access points in an office setting. Measurements of both types were gathered at 18 test positions (see \figurename~\ref{fig:real_rtt}) using an Android smartphone. The number of measurements $m$ at each position varies due to missing data and also depends on whether RSS, RTT, or both types of measurements are used. Consequently, a suitable reindexing of the senders $\vec{s}_j$ for $j = 1,\dotsc,m$ is needed.

The RSS is modeled using the log-distance path loss formula as in Section~\ref{sec:noise_distributions}, where we assume independent zero-mean Gaussian noise with standard deviation $\sigma_{RSS}=$ \SI{5}{\decibelmilliwatt}, i.e, $\epsilon_j \sim \mathcal{N}(0, 5)$ in \eqref{eq:log-distance_path_loss}. The parameters $(C_0)j$ and $\eta_j$ were estimated using linear least squares with measurements and positions from separate datasets.
The RSS measurements $C_j$ can then be converted to squared distance measurements $d_j^2$ using \eqref{eq:rss_distance_measurement}, and the weight matrix $\mat{W}$ becomes diagonal with elements
\begin{equation}
    w_{jj} = \left( \frac{5 \eta_j}{\sigma_{RSS} d_j^2 \log 10} \right)^2.
\end{equation}

RTT directly gives distance measurements $d_j$ similar to TOA, and we assume independent zero-mean Gaussian noise with standard deviation $\sigma_{RTT} =$ \SI{1}{\meter}, i.e., $\epsilon_j \sim \mathcal{N}(0, 1)$ in \eqref{eq:toa_model}. Similar to before, $\mat{W}$ becomes diagonal but with elements
\begin{equation}
    w_{jj} = \frac{1}{4 \sigma_{RTT}^2 d_j^2}.
\end{equation}

The 18 test positions were estimated using the proposed method (Alg.~\ref{alg:proposed}), with and without weighting, and a maximum likelihood estimate was found using local optimization initialized at the ground truth. The resulting mean errors for the different measurement types are shown in Table~\ref{tab:real_rtt_rss}.
As can be seen, the proposed method produced smaller errors than the ML estimate. 
This is not to be expected in general but is also not surprising for a small real dataset. The measurements could benefit the proposed method by chance and the noise might not be completely Gaussian.
Note also that the errors are larger in the unweighted case (when $\mat{W} = \mat{I}$). This is especially prominent when RSS measurements are used, highlighting the importance of suitable weights.
\figurename~\ref{fig:real_rtt} shows the estimated positions using only RTT with the weighting. The proposed and ML estimates are close compared to the error measured against the ground truth. Consequently, for this data, the noise has a greater impact on the positioning than the difference in the cost functions \eqref{eq:ls} and \eqref{eq:approx_cost}.

\begin{figure}
	\centering
	\includegraphics[width=\columnwidth]{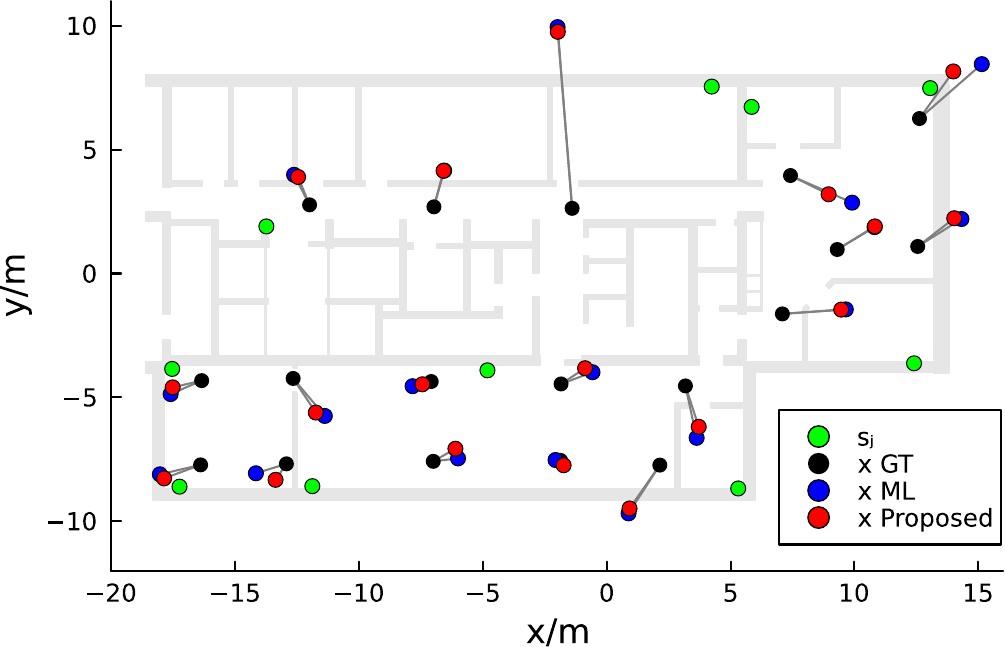}
	\caption{Estimated receiver positions $\vec{x}$ from the proposed method (red) and the ML estimate (blue) together with the ground truth positions (black) and ten senders (green) in an indoor office environment using Wi-Fi RTT measurements.}
	\label{fig:real_rtt}
\end{figure}

\begin{table}
	\centering
	\caption{Mean positioning error over 18 positions in an indoor office environment.}
	\label{tab:real_rtt_rss}
	\begin{tabular}{lccc}
		\toprule
		& Proposed & Proposed ($\mat{W}=\mat{I}$) & ML \\
		\midrule
		RTT     & \SI{1.7678}{\meter} & \SI{ 3.0386}{\meter} & \SI{2.0455}{\meter} \\
		RSS     & \SI{3.2663}{\meter} & \SI{16.7811}{\meter} & \SI{5.3340}{\meter} \\
		RTT+RSS & \SI{1.9395}{\meter} & \SI{11.6707}{\meter} & \SI{2.2386}{\meter} \\
		\bottomrule
	\end{tabular}
\end{table}

\section{Conclusion}
\label{sec:conclusion}
Trilateration is one of the fundamental problems in positioning and distance geometry with a wide range of applications. Still, we have presented a novel solution to the problem, not solely as a theoretical curiosity but as a competitive practical algorithm. Leveraging existing eigensolvers, the proposed method becomes fast, numerically stable, and easily implemented. The stability is further extended by the fact that we handle degenerate cases, scenarios that indeed occur in practice, especially when measurements are sparse.
The practicality and competitiveness of the proposed method have been demonstrated using synthetic and real experiments, where it is shown to be among the fastest and most accurate.

Trilateration is only one of the single-source localization problems, and future studies could investigate similar eigenvalue formulations for multilateration and triangulation, a work started in \cite{larsson_optimal_2019}. Fusing these problems into one common formulation, allowing for a combination of measurement types, would be an interesting problem. More work could also be done to classify all nontrivial degenerate cases of \eqref{eq:sls}, such as the one in \figurename~\ref{fig:degenerate_2d_2}.

\appendices

\section*{Acknowledgment}
This work was partially supported by the Wallenberg Artificial Intelligence, Autonomous Systems and Software Program (WASP) funded by the Knut and Alice Wallenberg Foundation, and the strategic research project ELLIIT. The authors also gratefully acknowledge Lund University Humanities Lab.

\section*{Code Availability}
Code and data can be found at \url{https://github.com/martinkjlarsson/trilateration}.

\ifCLASSOPTIONcaptionsoff
  \newpage
\fi

\bibliographystyle{IEEEtran}
\bibliography{IEEEabrv,ref}

\begin{thebibliography}{10}
\providecommand{\url}[1]{#1}
\csname url@samestyle\endcsname
\providecommand{\newblock}{\relax}
\providecommand{\bibinfo}[2]{#2}
\providecommand{\BIBentrySTDinterwordspacing}{\spaceskip=0pt\relax}
\providecommand{\BIBentryALTinterwordstretchfactor}{4}
\providecommand{\BIBentryALTinterwordspacing}{\spaceskip=\fontdimen2\font plus
\BIBentryALTinterwordstretchfactor\fontdimen3\font minus \fontdimen4\font\relax}
\providecommand{\BIBforeignlanguage}[2]{{%
\expandafter\ifx\csname l@#1\endcsname\relax
\typeout{** WARNING: IEEEtran.bst: No hyphenation pattern has been}%
\typeout{** loaded for the language `#1'. Using the pattern for}%
\typeout{** the default language instead.}%
\else
\language=\csname l@#1\endcsname
\fi
#2}}
\providecommand{\BIBdecl}{\relax}
\BIBdecl

\bibitem{petropoulos_chapter_2021}
\BIBentryALTinterwordspacing
E.~Ochin, ``Chapter 2 - {Fundamentals} of structural and functional organization of {GNSS},'' in \emph{{GPS} and {GNSS} {Technology} in {Geosciences}}, G.~p. Petropoulos and P.~K. Srivastava, Eds.\hskip 1em plus 0.5em minus 0.4em\relax Elsevier, 2021, pp. 21--49. [Online]. Available: \url{https://www.sciencedirect.com/science/article/pii/B978012818617600010X}
\BIBentrySTDinterwordspacing

\bibitem{li_review_2020}
\BIBentryALTinterwordspacing
M.~Li, F.~Jiang, and C.~Pei, ``\BIBforeignlanguage{en}{Review on {Positioning} {Technology} of {Wireless} {Sensor} {Networks}},'' \emph{\BIBforeignlanguage{en}{Wireless Personal Communications}}, vol. 115, no.~3, pp. 2023--2046, Dec. 2020. [Online]. Available: \url{https://doi.org/10.1007/s11277-020-07667-7}
\BIBentrySTDinterwordspacing

\bibitem{dunitz_distance_1990}
\BIBentryALTinterwordspacing
J.~D. Dunitz, ``\BIBforeignlanguage{en}{Distance geometry and molecular conformation},'' \emph{\BIBforeignlanguage{en}{Journal of Computational Chemistry}}, vol.~11, no.~2, pp. 265--266, Mar. 1990. [Online]. Available: \url{https://onlinelibrary.wiley.com/doi/10.1002/jcc.540110212}
\BIBentrySTDinterwordspacing

\bibitem{geng_3-2-1_1994}
\BIBentryALTinterwordspacing
Z.~J. Geng and L.~S. Haynes, ``A “3-2-1” kinematic configuration of a {Stewart} platform and its application to six degree of freedom pose measurements,'' \emph{Robotics and Computer-Integrated Manufacturing}, vol.~11, no.~1, pp. 23--34, Mar. 1994. [Online]. Available: \url{https://www.sciencedirect.com/science/article/pii/0736584594900043}
\BIBentrySTDinterwordspacing

\bibitem{ma_wi-fi_2022}
\BIBentryALTinterwordspacing
C.~Ma, B.~Wu, S.~Poslad, and D.~R. Selviah, ``Wi-{Fi} {RTT} {Ranging} {Performance} {Characterization} and {Positioning} {System} {Design},'' \emph{IEEE Transactions on Mobile Computing}, vol.~21, no.~2, pp. 740--756, Feb. 2022, conference Name: IEEE Transactions on Mobile Computing. [Online]. Available: \url{https://ieeexplore.ieee.org/abstract/document/9151400}
\BIBentrySTDinterwordspacing

\bibitem{luke_simple_2017}
\BIBentryALTinterwordspacing
D.~R. Luke, S.~Sabach, M.~Teboulle, and K.~Zatlawey, ``\BIBforeignlanguage{en}{A simple globally convergent algorithm for the nonsmooth nonconvex single source localization problem},'' \emph{\BIBforeignlanguage{en}{Journal of Global Optimization}}, vol.~69, no.~4, pp. 889--909, Dec. 2017. [Online]. Available: \url{https://doi.org/10.1007/s10898-017-0545-6}
\BIBentrySTDinterwordspacing

\bibitem{beck_iterative_2008}
\BIBentryALTinterwordspacing
A.~Beck, M.~Teboulle, and Z.~Chikishev, ``\BIBforeignlanguage{en}{Iterative {Minimization} {Schemes} for {Solving} the {Single} {Source} {Localization} {Problem}},'' \emph{\BIBforeignlanguage{en}{SIAM Journal on Optimization}}, vol.~19, no.~3, pp. 1397--1416, Jan. 2008. [Online]. Available: \url{http://epubs.siam.org/doi/10.1137/070698014}
\BIBentrySTDinterwordspacing

\bibitem{jyothi_solvit_2020}
R.~Jyothi and P.~Babu, ``{SOLVIT}: {A} {Reference}-{Free} {Source} {Localization} {Technique} {Using} {Majorization} {Minimization},'' \emph{IEEE/ACM Transactions on Audio, Speech, and Language Processing}, vol.~28, pp. 2661--2673, 2020, conference Name: IEEE/ACM Transactions on Audio, Speech, and Language Processing.

\bibitem{sirola_closed-form_2010}
N.~Sirola, ``Closed-form algorithms in mobile positioning: {Myths} and misconceptions,'' in \emph{Navigation and {Communication} 2010 7th {Workshop} on {Positioning}}, Mar. 2010, pp. 38--44.

\bibitem{chan_exact_2006}
Y.-T. Chan, H.~Yau Chin~Hang, and P.-c. Ching, ``Exact and approximate maximum likelihood localization algorithms,'' \emph{IEEE Transactions on Vehicular Technology}, vol.~55, no.~1, pp. 10--16, Jan. 2006, conference Name: IEEE Transactions on Vehicular Technology.

\bibitem{beck_exact_2008}
A.~Beck, P.~Stoica, and J.~Li, ``Exact and {Approximate} {Solutions} of {Source} {Localization} {Problems},'' \emph{IEEE Transactions on Signal Processing}, vol.~56, no.~5, pp. 1770--1778, May 2008, conference Name: IEEE Transactions on Signal Processing.

\bibitem{ismailova_penalty_2016}
\BIBentryALTinterwordspacing
D.~Ismailova and W.-S. Lu, ``\BIBforeignlanguage{en}{Penalty convex-concave procedure for source localization problem},'' in \emph{\BIBforeignlanguage{en}{2016 {IEEE} {Canadian} {Conference} on {Electrical} and {Computer} {Engineering} ({CCECE})}}.\hskip 1em plus 0.5em minus 0.4em\relax Vancouver, BC, Canada: IEEE, May 2016, pp. 1--4. [Online]. Available: \url{http://ieeexplore.ieee.org/document/7726815/}
\BIBentrySTDinterwordspacing

\bibitem{lipp_variations_2016}
\BIBentryALTinterwordspacing
T.~Lipp and S.~Boyd, ``\BIBforeignlanguage{en}{Variations and extension of the convex–concave procedure},'' \emph{\BIBforeignlanguage{en}{Optimization and Engineering}}, vol.~17, no.~2, pp. 263--287, Jun. 2016. [Online]. Available: \url{https://doi.org/10.1007/s11081-015-9294-x}
\BIBentrySTDinterwordspacing

\bibitem{cheung_least_2004}
K.~W. Cheung, H.~C. So, W.~Ma, and Y.~T. Chan, ``Least squares algorithms for time-of-arrival-based mobile location,'' \emph{IEEE Transactions on Signal Processing}, vol.~52, no.~4, pp. 1121--1130, Apr. 2004, conference Name: IEEE Transactions on Signal Processing.

\bibitem{chen_achieving_2013}
\BIBentryALTinterwordspacing
S.~Chen and K.~C. Ho, ``Achieving {Asymptotic} {Efficient} {Performance} for {Squared} {Range} and {Squared} {Range} {Difference} {Localizations},'' \emph{IEEE Transactions on Signal Processing}, vol.~61, no.~11, pp. 2836--2849, Jun. 2013, conference Name: IEEE Transactions on Signal Processing. [Online]. Available: \url{https://ieeexplore.ieee.org/document/6487416/?arnumber=6487416}
\BIBentrySTDinterwordspacing

\bibitem{ismailova_improved_2015}
D.~Ismailova and W.-S. Lu, ``Improved least-squares methods for source localization: {An} iterative {Re}-weighting approach,'' in \emph{2015 {IEEE} {International} {Conference} on {Digital} {Signal} {Processing} ({DSP})}, Jul. 2015, pp. 665--669, iSSN: 2165-3577.

\bibitem{chen_reaching_2014}
\BIBentryALTinterwordspacing
S.~Chen and K.~C. Ho, ``Reaching asymptotic efficient performance for squared processing of range and range difference localizations in the presence of sensor position errors,'' in \emph{2014 {IEEE} {International} {Conference} on {Acoustics}, {Speech} and {Signal} {Processing} ({ICASSP})}, May 2014, pp. 1419--1423, iSSN: 2379-190X. [Online]. Available: \url{https://ieeexplore.ieee.org/document/6853831/?arnumber=6853831}
\BIBentrySTDinterwordspacing

\bibitem{larsson_accuracy_2010}
\BIBentryALTinterwordspacing
E.~Larsson and D.~Danev, ``\BIBforeignlanguage{en}{Accuracy {Comparison} of {LS} and {Squared}-{Range} {LS} for {Source} {Localization}},'' \emph{\BIBforeignlanguage{en}{IEEE Transactions on Signal Processing}}, vol.~58, no.~2, pp. 916--923, Feb. 2010. [Online]. Available: \url{http://ieeexplore.ieee.org/document/5233795/}
\BIBentrySTDinterwordspacing

\bibitem{zhou_closed-form_2011}
\BIBentryALTinterwordspacing
Y.~Zhou, ``\BIBforeignlanguage{en}{A closed-form algorithm for the least-squares trilateration problem},'' \emph{\BIBforeignlanguage{en}{Robotica}}, vol.~29, no.~3, pp. 375--389, May 2011, publisher: Cambridge University Press. [Online]. Available: \url{https://www.cambridge.org/core/journals/robotica/article/closedform-algorithm-for-the-leastsquares-trilateration-problem/FC7B1E4BAADD781FEE559DD304A31409}
\BIBentrySTDinterwordspacing

\bibitem{thomas_revisiting_2005}
F.~Thomas and L.~Ros, ``Revisiting trilateration for robot localization,'' \emph{IEEE Transactions on Robotics}, vol.~21, no.~1, pp. 93--101, Feb. 2005, conference Name: IEEE Transactions on Robotics.

\bibitem{manolakis_efficient_1996}
D.~E. Manolakis, ``Efficient solution and performance analysis of 3-{D} position estimation by trilateration,'' \emph{IEEE Transactions on Aerospace and Electronic Systems}, vol.~32, no.~4, pp. 1239--1248, Oct. 1996, conference Name: IEEE Transactions on Aerospace and Electronic Systems.

\bibitem{coope_reliable_2000}
\BIBentryALTinterwordspacing
I.~D. Coope, ``\BIBforeignlanguage{en}{Reliable computation of the points of intersection of {$n$} spheres in {$R^n$}},'' \emph{\BIBforeignlanguage{en}{ANZIAM Journal}}, vol.~42, pp. C461--C477, Dec. 2000. [Online]. Available: \url{https://journal.austms.org.au/ojs/index.php/ANZIAMJ/article/view/608}
\BIBentrySTDinterwordspacing

\bibitem{caffery_new_2000}
J.~Caffery, ``A new approach to the geometry of {TOA} location,'' in \emph{Vehicular {Technology} {Conference} {Fall} 2000. {IEEE} {VTS} {Fall} {VTC2000}. 52nd {Vehicular} {Technology} {Conference} ({Cat}. {No}.{00CH37152})}, vol.~4, Sep. 2000, pp. 1943--1949 vol.4, iSSN: 1090-3038.

\bibitem{navidi_statistical_1998}
\BIBentryALTinterwordspacing
W.~Navidi, W.~S. Murphy, and W.~Hereman, ``\BIBforeignlanguage{en}{Statistical methods in surveying by trilateration},'' \emph{\BIBforeignlanguage{en}{Computational Statistics \& Data Analysis}}, vol.~27, no.~2, pp. 209--227, Apr. 1998. [Online]. Available: \url{https://www.sciencedirect.com/science/article/pii/S0167947397000534}
\BIBentrySTDinterwordspacing

\bibitem{stoica_lecture_2006}
P.~Stoica and J.~Li, ``Lecture {Notes} - {Source} {Localization} from {Range}-{Difference} {Measurements},'' \emph{IEEE Signal Processing Magazine}, vol.~23, no.~6, pp. 63--66, Nov. 2006, conference Name: IEEE Signal Processing Magazine.

\bibitem{zekavat_handbook_2012.ch2}
\BIBentryALTinterwordspacing
H.~C. So, \emph{Source Localization: Algorithms and Analysis}.\hskip 1em plus 0.5em minus 0.4em\relax John Wiley \& Sons, Ltd, 2011, ch.~2, pp. 25--66. [Online]. Available: \url{https://onlinelibrary.wiley.com/doi/abs/10.1002/9781118104750.ch2}
\BIBentrySTDinterwordspacing

\bibitem{larsson_optimal_2019}
M.~Larsson, V.~Larsson, K.~{\AA}str{\"o}m, and M.~Oskarsson, ``Optimal {Trilateration} {Is} an {Eigenvalue} {Problem},'' in \emph{{ICASSP} 2019 - 2019 {IEEE} {International} {Conference} on {Acoustics}, {Speech} and {Signal} {Processing} ({ICASSP})}, May 2019, pp. 5586--5590, iSSN: 2379-190X.

\bibitem{larsson_localization_2022}
M.~Larsson, ``Localization using {Distance} {Geometry}: {Minimal} {Solvers} and {Robust} {Methods} for {Sensor} {Network} {Self}-{Calibration},'' Doctoral {Thesis} (compilation), Mathematics Centre for Mathematical Sciences Lund University Lund, Lund, 2022, {ISBN}: 9789180393775.

\bibitem{sharp_wireless_2019}
\BIBentryALTinterwordspacing
I.~Sharp and K.~Yu, \emph{Wireless {Positioning}: {Principles} and {Practice}}, ser. Navigation: {Science} and {Technology}.\hskip 1em plus 0.5em minus 0.4em\relax Singapore: Springer, 2019. [Online]. Available: \url{http://link.springer.com/10.1007/978-981-10-8791-2}
\BIBentrySTDinterwordspacing

\bibitem{cost_action_231_1999}
E.~Commission, D.-G. for~the Information~Society, and Media, \emph{COST Action 231 : Digital mobile radio towards future generation systems: Final Report}.\hskip 1em plus 0.5em minus 0.4em\relax Publications Office, 1999.

\bibitem{more_generalizations_1993}
\BIBentryALTinterwordspacing
J.~J. More, ``\BIBforeignlanguage{en}{Generalizations of the trust region problem},'' \emph{\BIBforeignlanguage{en}{Optimization Methods and Software}}, vol.~2, no. 3-4, pp. 189--209, Jan. 1993. [Online]. Available: \url{http://www.tandfonline.com/doi/abs/10.1080/10556789308805542}
\BIBentrySTDinterwordspacing

\bibitem{arnoldi_principle_1951}
\BIBentryALTinterwordspacing
W.~E. Arnoldi, ``\BIBforeignlanguage{en}{The principle of minimized iterations in the solution of the matrix eigenvalue problem},'' \emph{\BIBforeignlanguage{en}{Quarterly of Applied Mathematics}}, vol.~9, no.~1, pp. 17--29, 1951. [Online]. Available: \url{https://www.ams.org/qam/1951-09-01/S0033-569X-1951-42792-9/}
\BIBentrySTDinterwordspacing

\bibitem{lehoucq_arpack_1998}
\BIBentryALTinterwordspacing
R.~B. Lehoucq, D.~C. Sorensen, and C.~Yang, \emph{{ARPACK} {Users}' {Guide}}.\hskip 1em plus 0.5em minus 0.4em\relax Society for Industrial and Applied Mathematics, 1998, \_eprint: https://epubs.siam.org/doi/pdf/10.1137/1.9780898719628. [Online]. Available: \url{https://epubs.siam.org/doi/abs/10.1137/1.9780898719628}
\BIBentrySTDinterwordspacing

\bibitem{adachi_solving_2017}
\BIBentryALTinterwordspacing
S.~Adachi, S.~Iwata, Y.~Nakatsukasa, and A.~Takeda, ``\BIBforeignlanguage{en}{Solving the {Trust}-{Region} {Subproblem} {By} a {Generalized} {Eigenvalue} {Problem}},'' \emph{\BIBforeignlanguage{en}{SIAM Journal on Optimization}}, vol.~27, no.~1, pp. 269--291, Jan. 2017. [Online]. Available: \url{http://epubs.siam.org/doi/10.1137/16M1058200}
\BIBentrySTDinterwordspacing

\bibitem{adachi_eigenvalue-based_2019}
\BIBentryALTinterwordspacing
S.~Adachi and Y.~Nakatsukasa, ``\BIBforeignlanguage{en}{Eigenvalue-based algorithm and analysis for nonconvex {QCQP} with one constraint},'' \emph{\BIBforeignlanguage{en}{Mathematical Programming}}, vol. 173, no.~1, pp. 79--116, Jan. 2019. [Online]. Available: \url{https://doi.org/10.1007/s10107-017-1206-8}
\BIBentrySTDinterwordspacing

\bibitem{coey_solving_2022}
C.~Coey, L.~Kapelevich, and J.~P. Vielma, ``Solving natural conic formulations with {H}ypatia.jl,'' \emph{INFORMS Journal on Computing}, vol.~34, no.~5, pp. 2686--2699, 2022.

\bibitem{clarabel}
\BIBentryALTinterwordspacing
P.~Goulart and Y.~Chen, ``Clarabel,'' accessed: 2024-01-16. [Online]. Available: \url{https://oxfordcontrol.github.io/ClarabelDocs}
\BIBentrySTDinterwordspacing

\bibitem{bezanzon_julia_2017}
\BIBentryALTinterwordspacing
J.~Bezanson, A.~Edelman, S.~Karpinski, and V.~B. Shah, ``Julia: A fresh approach to numerical computing,'' \emph{SIAM Review}, vol.~59, no.~1, pp. 65--98, 2017. [Online]. Available: \url{https://doi.org/10.1137/141000671}
\BIBentrySTDinterwordspacing

\end{thebibliography}

\end{document}